\documentclass[12pt,a4paper]{amsart}         
\usepackage[margin=3cm]{geometry}

\usepackage{amsmath,amsfonts,amssymb,amsthm,mathrsfs,epsfig,epstopdf,url,array}
\usepackage{float}
\restylefloat{table}

\theoremstyle{definition}
\newtheorem{definition}{Definition}[section]

\theoremstyle{plain}
\newtheorem{theorem}[definition]{Theorem}

\theoremstyle{remark}

\newcommand{\trf}[2]{\tr_{\mathbb{F}_{#1}/\mathbb{F}_{#2}}}
\DeclareMathOperator{\tr}{Tr}
\DeclareMathOperator{\Frob}{Frob}
\DeclareMathOperator{\supp}{supp}

\title[Subfield Subcodes of Hermitian Codes]{On The Dimension of The Subfield Subcodes of 1-Point Hermitian Codes}

\author{Sabira El Khalfaoui}
\address{Bolyai Institute \\
	University of Szeged \\
	Aradi v\'ertan\'uk tere 1\\
	H-6720 Szeged, Hungary}
\email{sabira@math.u-szeged.hu}

\author{G\'abor P. Nagy}
\address{Department of Algebra \\
	Budapest University of Technology and Economics\\
	Egry J\'ozsef utca 1\\
	H-1111 Budapest, Hungary}
\address{Bolyai Institute \\
	University of Szeged \\
	Aradi v\'ertan\'uk tere 1\\
	H-6720 Szeged, Hungary}
\email{nagyg@math.u-szeged.hu}

\thanks{Support provided from the National Research, Development and Innovation Fund of Hungary, financed under the 2018-1.2.1-NKP funding scheme, within the SETIT Project (2018-1.2.1-NKP-2018-00004). Partially supported by NKFIH-OTKA Grants 119687 and 115288.}


\begin{document}

\begin{abstract}
Subfield subcodes of algebraic-geometric codes are good candidates for the use in post-quantum cryptosystems, provided their true parameters such as dimension and minimum distance can be determined. In this paper we present new values of the true dimension of subfield subcodes of $1$--point Hermitian codes, including the case when the subfield is not binary.
\end{abstract}

\maketitle

\section{Introduction}

The oldest and best known proposal for post-quantum cryptography schemes are the cryptosystems due to McEliece and Niederreiter. Their security is based on the NP-completeness of the decoding of binary linear codes. Hence, an essential ingredient of their schemes is a binary linear code $C$ which has an efficient decoding algorithm and which cannot be distinguished from the random linear code. McEliece originally proposed the class of extended binary Goppa codes, which are subfield subcodes of the generalized Reed-Solomon codes. Recently, some other classes of codes have been proposed as well, such as LDPC codes and algebraic-geometric codes over larger fields. However, these classes turned out to have serious security flows, see \cite{baldi2009ldpc,bernstein2008attacking,li1994equivalence,LoidreauSendrier,wieschebrink2010cryptanalysis}. For the background of code based cryptography we refer to  \cite{menezes2013applications,stepanov2012codes,stichtenoth2009algebraic}, for quantum attacks see \cite{yan2013quantum}, and on digital signature schemes based on the Niederreiter scheme see \cite{courtois2001achieve}. 

The Berlekamp-Massey algorithm \cite{elia1999decoding} and its variants provide an efficient decoding for Reed-Solomon codes, which can be used to decode subfield subcodes of generalized Reed-Solomon codes, as well. For the binary linear code $C$ in use, the error correcting bound is determined by these algorithms. Beyond this bound, list-decoding methods are known, cf. \cite{augot2011list,bernstein2011list,misoczki2009compact}. Therefore, it is an important problem to find the true minimum distance and the true dimension of subfield subcodes of generalized Reed-Solomon codes, cf. \cite{Del} and the series of papers \cite{V98,V01,V05}. 
The class of algebraic-geometry (AG) codes was introduced by V.D.~Goppa. This class is a natural generalization of Reed-Solomon codes. The famous Riemann-Roch Theorem provides theoretical bounds for the dimension and minimum distance of AG codes. The ideas of the Berlekamp-Massey algorithm can be used to design efficient decoding algorithms up to the half of the designed minimum distance of AG codes, and beyond \cite{hoholdt1995decoding,pellikaan1993efficient,sakata1995generalized}. Hence, the subfield subcodes of AG codes are also good candidates for the McEliece and Niederreiter cryptosystems. The determination of the true dimension and the true minimum distance of the subfield subcodes of AG codes seems to be a hard problem, the attempts so far focused mainly at 1-point Hermitian codes and their subcodes, with some further restrictions on the parameters \cite{pinero2014subfield,van1990true,van1991dimension}. 

In this paper, we prove new results on the true dimension of the subfield subcodes of 1-point Hermitian codes. Our approach deals also with non-binary subfields. The paper is structured as follows. In section \ref{sec:hermitiancurves}, we describe the backgrounds with some important properties of Hermitian curves, their function fields and Riemann-Roch spaces. In section \ref{sec:agcodes}, we present AG codes and $1$-point Hermitian codes. Section \ref{sec:subfieldsubcodes} summarizes the definition of the subfield subcodes of AG codes and techniques used to improve the bounds on the dimensions of subfield subcodes of Reed-Solomon codes, these techniques include Delsarte's seminal result on subfield subcodes and trace codes. Section \ref{sec:mainresult} is dedicated to prove our result concerning the true dimension of the subfield subcodes of 1-point Hermitian codes for specific parameters.

\section{Hermitian curves, their divisors and Riemann-Roch spaces}
\label{sec:hermitiancurves}

Our notation and terminology on algebraic plane curves over finite fields, their function fields, divisors and Riemann-Roch spaces are standard, see for instance \cite{hirschfeld2008algebraic,menezes2013applications,stichtenoth2009algebraic}. 

Let $PG(2,\mathbb{F}_{q^2})$ be the projective plane over the finite field of order $q^2$ equipped with homogeneous coordinates $(X,Y,Z)$. The Hermitian curve in its canonical form is the non-singular plane curve $\mathscr{H}_{q}$ with equation $Y^qZ+YZ^q=X^{q+1}$. The genus of $\mathscr{H}_q$ equals $g=q(q-1)/2$ and the set $\mathscr{H}_q(\mathbb{F}_{q^2})$ of $\mathbb{F}_{q^2}$-rational points, that is, its points with coordinates over $\mathbb{F}_{q^2}$ has size $q^3+1$. It is also useful to regard $PG(2,\mathbb{F}_{q^2})$ as the projective closure of the affine plane $AG(2,\mathbb{F}_{q^2})$ with respect to the line $Z=0$ at infinity, so that the $\mathscr{H}_q$ has affine equation $Y^q+Y=X^{q+1}$. In particular, $\mathscr{H}_q$ has just one point at infinity, namely $(0,1,0)$, denoted by $P_\infty$. We remark that Hermitian curves have the maximum number of rational points allowed by the Hasse-Weil bound \cite[Theorem 9.18]{hirschfeld2008algebraic}, \cite[Theorem 9.10]{menezes2013applications}.

The action of the Frobenius automorphism $\Frob_{q^2}:x\mapsto x^{q^2}$ can be extended to the points of $\mathscr{H}_q$ by applying $\Frob_{q^2}$ on the coordinates. We denote the extended action by $\Frob_{q^2}$ as well. A point $P$ of $\mathscr{H}_q$ is $\mathbb{F}_{q^2}$-rational if and only if $P=\Frob_{q^2}(P)$. 

As usual, we also look at the curve $\mathscr{H}_q$ as the curve defined over the algebraic closure $\bar{\mathbb{F}}_{q^2}$. Then, there is a one-to-one correspondence between the points of $\mathscr{H}_q$ and the places of the function field $\bar{\mathbb{F}}_{q^2}(\mathscr{H}_q)$ of $\mathscr{H}_q$. 

For a divisor $D=\lambda_1P_1+\cdots+\lambda_kP_k$ with $P_1,\ldots,P_k \in \mathscr{H}_q$ and integers $\lambda_1,\ldots,\lambda_k$, its Frobenius image is  \[\Frob_{q^2}(D)=\lambda_1 \Frob_q(P_1)  + \cdots + \lambda_k \Frob_q(P_k).\]
A divisor $D$ is $\mathbb{F}_{q^2}$-rational if $D=\Frob_{q^2}(D)$. In particular, if $P_1,\ldots,P_k$ are in $\mathscr{H}_q(\mathbb{F}_{q^2})$ then $D$ is $\mathbb{F}_{q^2}$-rational, but the converse does not hold in general. The degree of $D$ is $\deg(D)=\lambda_1+\cdots+\lambda_k$, while the support of $D$ is the set of points $P_i$ with $\lambda_i\neq 0$. 

For a non-zero function $f$ in the function field $\bar{\mathbb{F}}_{q^2}(\mathscr{H}_q)$ and a point $P$, $v_P(f)$ stands for the order of $f$ at $P$. If $v_P(f)>0$ then $P$ is a zero of $f$, while if $v_P(f)<0$, then $P$ is a pole of $f$ with multiplicity $-v_P(f)$. The principal divisor of a non-zero function $f$ is $(f) = \sum_P v_P(f)P$. 

For an $\mathbb{F}_{q^2}$-rational divisor $D$, the Riemann-Roch space $\mathscr{L}(D)$ is the vector space
\[\mathscr{L}(D) = \{ f \in {\mathbb{F}}_{q^2}(\mathscr{H}_q) \mid (f)\succcurlyeq -D \}.\]
For the dimension $\ell(D)$ of $\mathscr{L}(D)$ the Riemann-Roch Theorem states
\[\ell(D)=\deg(D)+1-g+\ell(W-D),\]
where $W$ is a canonical divisor of the Hermitian curve, for example $W=(q-2)(q+1)P_\infty$ is such a canonical divisor. Moreover, if $\deg(D)>2g-2$ then the Riemann-Roch Theorem reads $\ell(D)=\deg(D)+1-g$. Let $s$ be a positive integer and $D=sP_\infty$ a $1$-point divisor of $\mathscr{H}_q$. Then the set
\[
\left\lbrace x^i y^j  \mid 0\leq i \leq q^2 -1, \quad 0\leq j \leq q-1, \quad v_{P_\infty}(x^i y^j ) \leq s\right\rbrace
\]
of functions forms a basis of the Riemann-Roch space  $\mathscr{L}(sP_{\infty})$, see \cite[Theorem 10.4]{menezes2013applications}. Notice that $v_{P_\infty}(x)=q$, $v_{P_\infty}(y)=q+1$, and hence the order of $x^iy^j$ at $P_\infty$ is
\begin{equation} \label{eq:vPinfty}
v_{P_\infty}(x^iy^j)=qi+(q+1)j.
\end{equation}

\section{Algebraic geometry codes (briefly AG codes)}
\label{sec:agcodes}

Algebraic geometry codes are a type of linear error correcting block codes, arising from algebraic curves defined over a finite field, see \cite{stichtenoth2009algebraic}. Here we outline the construction when the underlying curve is the Hermitian curve $\mathscr{H}_q$. 

Fix a divisor $D=P_1+...+P_n$ where all $P_i$ are pairwise distinct $\mathbb{F}_{q^2}$--rational points of $\mathscr{H}_q$. Also, take another $\mathbb{F}_{q^2}$-rational divisor $G$ whose support is disjoint from  $\supp\, D$. The functional AG code $C_L(D,G)$ associated with the divisors $D$ and $G$ is a subspace of the vector space $\mathbb{F}_{q^2}^n$, and defined by
\[
C_L(D,G)=\left\lbrace \left( f(P_1),...,f(P_n) \right) \,|\, f \in \mathscr{L}(G)  \right\rbrace \subseteq \mathbb{F}_{q^2}^{n}. 
\]
In other words, $C_L(D,G)$ is the image of $\mathscr{L}(G)$ under the evaluation map
\[
\mathscr{L}(G)\ni f \mapsto (f(P_1),...,f(P_n))\in \mathbb{F}_{q^2}^n.
\]

Indeed, determining the functions field and the divisors in a pertinent way can make Reed-Solomon codes viewed as particular AG codes, see \cite[Section 2.3]{stichtenoth2009algebraic}. The most fascinating feature of AG codes is that the Riemann-Roch Theorem determines its dimension $k$ and provides a useful bound for its minimum distance $d$.

\begin{theorem}[{\cite[Theorem 10.1]{menezes2013applications}}]
\label{thm:C_Lparam}
	$C_L(D,G)$ is a linear $\left[ n,k,d \right] $ code over $\mathbb{F}$ with parameters:
	
	\begin{itemize}
		\item $k = \ell (G) - \ell (G-D)$,
		\item $d \geq n - \deg G$.
	\end{itemize}
\end{theorem}

Notice that the condition $n>\deg G$ implies the evaluation map $\mathscr{L}(G)\to \mathbb{F}^n$ to be injective. If $n\leq \deg G$, then it is possible that $C_L(D,G)$ has dimension less than $n$ and positive true minimum distance. However, this case cannot be described only by the Riemann-Roch Theorem. 

By using the differential space $\Omega(G)$ instead of the Riemann-Roch space $\mathscr{L}(G)$, one can define another AG codes, namely the differential AG codes $C_{\Omega}(D,G)$. It should be noted that the differential code $C_{\Omega}(D,G)$ is the dual of the functional AG code $C_L(D,G)$.

The main result of this paper deals with $1$-point Hermitian codes. Let $n=q^3$ and the divisor $D=P_1 + P_2 +...+P_n$ be the sum of $\mathbb{F}_{q^2}$--rational affine points of $\mathscr{H}_q$. For a positive integer $s$, we denote by $\mathcal{H}(q^2, s)$ the $1$-point functional AG code $C_L(D,sP_{\infty})$. This has length $n= q^3$. If  $2g-2 < s < n$, then the dimension of $\mathcal{H}(q^2,s) $ is $k= s - g +1$ which is equal to the dimension of the Riemann-Roch space $\mathscr{L}(sP_{\infty})$. Under these assumptions, we have equality in Theorem \ref{thm:C_Lparam}, hence the minimum distance of $\mathcal{H}(q^2,s)$ is $d= q^3 -s$.

\begin{theorem}[Dual codes {\cite[Theorem 10.5]{menezes2013applications}}]
	For $s>0$ and $\tilde{s}= q^3 + q^2 - q-2-s$, the codes $\mathcal{H}(q^2, s)$ and $\mathcal{H}(q^2, \tilde{s})$ are dual to each other. 
	In particular, if $q$ is even and $s= (q^3 +q^2-q-2)/2$, then the code $\mathcal{H}(q^2, s)$ is self-dual.
\end{theorem}

\section{Subfield subcodes of linear codes}
\label{sec:subfieldsubcodes}

Let $\mathbb{F}_r$ be a subfield of a finite field $\mathbb{F}_l$, that is, $l=r^h$ for a positive integer $h$. Let $C$ be a linear $\left[n,k,d \right] $ code over $\mathbb{F}_l$. The subfield subcode $C|\mathbb{F}_r$ consists of all codewords of $C$ whose coordinates are in $\mathbb{F}_r$, that is, 
\[
		C|\mathbb{F}_r = C \cap \mathbb{F}_r^n.
\]

This is a linear $\left[ n,k_0,d_0\right] $ code over $\mathbb{F}_r$ with $d \leq d_0 \leq n$ and $n-k \leq n-k_0\leq h(n-k)$. Therefore for the dimension over $\mathbb{F}_r$
\begin{equation} \label{eq:firstbound}
k_0\geq n-h(n-k).
\end{equation}
A parity check matrix of $C$ over $\mathbb{F}_l$ yields at most $h(n-k)$ linearly independent parity check equations over $\mathbb{F}_r$ for the subfield subcode $C|\mathbb{F}_r$.

In general the true minimum distance of a subfield subcodes  is bigger than the minimum distance of the original code. This makes the subfield subcodes very important, especially in the binary case $r=2$, see \cite[Theorem~4]{Del}.

\medskip

The trace polynomial $\tr(X) \in \mathbb{F}_r[x]$ with respect to $\mathbb{F}_l$ is given by
\[
\trf{l}{r}(x)= x + x^r + ... + x^{r^{h-1}}.
\]
Clearly, the trace polynomial determines the $\mathbb{F}_r$-linear trace map $\mathbb{F}_l\to \mathbb{F}_r$. For a linear code $C$ over $\mathbb{F}_l$, Delsarte defined the trace code $\tr(C)=\trf{l}{r}(C)$ by
\[\tr(C)= \left\lbrace \left( \trf{l}{r}(c_1),...,\trf{l}{r}(c_n)\right)\, \vert \, \left( c_1,\dots,c_n\right)   \in C \right\rbrace,\] 
and showed that $\tr(C)$ is a linear $\left[ n,k_1,d_1\right] $ code over $\mathbb{F}_r$, with $1 \leqslant d_1 \leqslant d$ and $k\leqslant k_1 \leqslant hk$. As for subfield subcodes, the most useful case occurs for $r=2$. 

The following important result by Delsarte relates the class of subfield subcodes to trace codes:
\begin{theorem}[Delsarte \cite{Del}]
Let $C$ be a linear code over an extension field $\mathbb{F}_l$ of $\mathbb{F}_r$. Then $(C|\mathbb{F}_r)^{\bot} = \tr(C^{\bot})$ holds.
\end{theorem}
In \cite{V05}, V\'eron pointed out that Delsarte's theorem can be used to compute from \eqref{eq:firstbound} the exact dimension 
\begin{equation} \label{eq:dels_veron_bound}
k_0 = n-h(n-k) +\dim_{\mathbb{F}_r}\ker(\tr)
\end{equation}
of the subfield subcode. 

\section{Main result}
\label{sec:mainresult}

With the above notation, let $l=q^2$ and $h=2m$. As before, let $\mathbb{F}_r$ be a subfield of $\mathbb{F}_{q^2}$, $q=r^m$, $s$ be a positive integer and $D$ be the sum of affine points of the Hermitian curve $X^{q+1}=Y+Y^q$ over the finite field $\mathbb{F}_{q^2}$. Define $C_{q,r}(s)$ to be the subfield subcode $\mathcal{H}(q^2,s)|{\mathbb{F}_r}$ of the $1$-point Hermitian code $\mathcal{H}(q^2,s)$. 

In \cite{pinero2014subfield}, an algorithm for $\dim C_{q,r}(s)$ is presented. Using this algorithm, the authors explicitly compute the dimension of $C_{4,2}(s)$ for each $s=0,\ldots,71$. 

From \cite[Proposition 3.2]{van1991dimension}, 
\[\dim(\tr(\mathcal{H}(q^2,q)))=2m+1,\]
where $q=2^m$. In our notation, this reads
\[\dim C_{q,r}(q^3+q^2-2q-2)=q^3-(2m+1).\]
In particular, $\dim C_{4,2}(70)=59$, which is confirmed by \cite[Table 2]{pinero2014subfield}. In the same table, we find $\dim C_{4,2}(s) = 1$ for $s=0,\ldots,31$ and $\dim C_{4,2}(32)=5$. These values for $\dim C_{4,2}(s)$ are particular cases of the general formula given by the following theorem. 

	\begin{theorem}
	Let $C_{q,r}(s)$ be a subfield subcode of the Hermitian code $\mathcal{H}(q^2,s)$, where $q = r^m$ is a prime power. Then
	\[
	\dim C_{q,r}(s) = \left\{
	\begin{array}{ll}
	1 & \mbox{for } 0 \leq s <\frac{q^3}{r} \\
	2m +1  & \mbox{for } s=\frac{q^3}{r}
	\end{array}
	\right.
	\]		  
	\end{theorem}
\begin{proof}
Since the constant polynomials are in $\mathscr{L}(sP_\infty)$ for all $s\geq 0$, we have $\dim C_{q,r}(s)\geq 1$. We first show that $\dim C_{q,r}(s) = 1$ for $0 \leq s <\frac{q^3}{r}$. Fix an integer $0<s<\frac{q^3}{r}$ and take an arbitrary element $(c_1,\ldots,c_{q^3}) \in C_{q,r}(s)$. Then there is an element $f\in \mathscr{L}(sP_\infty)$ such that for all $i=1,\ldots,q^3$, one has $c_i=f(P_i)\in \mathbb{F}_r$. There is an element $\gamma \in \mathbb{F}_r$ such that $c_i=\gamma$ for at least $q^3/r$ indices $i$. In other words, $f-\gamma \in \mathscr{L}(sP_\infty)$ has at least $q^3/r$ zeros on the Hermitian curve $\mathscr{H}_q$. ( This follows from the fact that for a positive divisor $G$, a non-zero element of $\mathscr{L}(G)$ cannot have more than $\deg G$ zeros.)
Therefore, $f-\gamma$ must be the constant zero polynomial, and $c_i=\gamma$ for all $i$. In particular, $C_{q,r}(s)$ consists of the constant vectors. 

Now, we suppose that $s = q^3/r$. Recall that 
\[\tr(X)=X+X^r+\cdots+X^{r^{2m-1}}\]
is the trace polynomial of $\mathbb{F}_{q^2}$ over $\mathbb{F}_r$. We define the polynomial
\[
f_{d,\alpha}(X) = d+ \tr (\alpha X)
\] 
where $d\in \mathbb{F}_r$, $\alpha \in \mathbb{F}_{q^2}$. As a polynomial in one variable, $f_{d,\alpha}$ maps $\mathbb{F}_{q^2}$ to $\mathbb{F}_r$. For a point $P$ with affine coordinates $(x,y)$, we write $f_{d,\alpha}(P) = f_{d,\alpha}(x)$. For the $\mathbb{F}_{q2}$-rational points $P_i (a_i, b_i)$, $i=1,\dots,q^3$, we have $f_{d,\alpha}(P_i) \in \mathbb{F}_r$. In other words, the evaluation vector
\[\mathbf{c}_{d,\alpha} = (f_{d,\alpha}(P_1),\ldots,f_{d,\alpha}(P_{q^3})) \in \mathbb{F}_r^n.\]
We claim that $f_{d,\alpha}(x) \in  \mathscr{L}\left(\frac{q^3}{r} P_{\infty}\right)$. In fact, by \eqref{eq:vPinfty},
\[v_{P_\infty}(x^{r^k})=qr^k,\]
which is at most $qr^{2m-1}=q^3/r$ for $k\leq 2m-1$. Hence, all monomials of $f_{d,\alpha}(x)$ are in  $\mathscr{L}\left(\frac{q^3}{r} P_{\infty}\right)$, and the claim follows. 

From the last two properties of $f_{d,\alpha}$ follows that the evaluation vector $\mathbf{c}_{d,\alpha} \in C_{q,r}(q^3/r)$. Since the map $(d,\alpha) \mapsto \mathbf{c}_{d,\alpha}$ is linear over $\mathbb{F}_r$, and injective, we have $\dim C_{q,r}(q^3/r)\geq 2m+1$. 

In the last step we show that the elements $\mathbf{c}_{d,\alpha}$ exhaust the subfield subcode $C_{q,r}(q^3/r)$. 

Take an element $g\in \mathscr{L}\left(\frac{q^3}{r} P_{\infty}\right)$ whose evaluation vector 
\[(g(P_1),\ldots,g(P_{q^3}))\in \mathbb{F}_r^n.\]
We can reduce the high $y$-degree terms by the Hermitian equation $x^{q+1}=y+y^q$. Thus, we can write $g$ in this form: 
\[g(x,y)=\sum\limits_{j<q} a_{i,j} x^i y^j. \]
By \eqref{eq:vPinfty}, the order of $x^iy^j$ at $P_\infty$ satisfies $v_{P_\infty} (x^iy^j)\equiv j \pmod q$. Therefore, if $j \leq q-1$ then the order $v_{P_\infty} (x^iy^j)$ determines $i$ and $j$ uniquely. Hence, different terms of $g=\sum\limits_{j\leq q-1} a_{i,j} x^i y^j$ have different orders at $P_\infty$. 
The order of $g$ at $P_\infty$ is 
\[\mathit{v}_{P_{\infty}} (g)= \mathit{v}_{P_{\infty}} \left( \sum a_{i,j} x^i y^j  \right) = \max\limits_{a_{i,j}\neq 0} \left( \mathit{v}_{P_{\infty}}(x^iy^j)\right), \]
where the last equality holds since the orders $v_{P_\infty} (x^iy^j)$ are different. If $g\in \mathscr{L}\left( \left( \frac{q^3}{r}-1 \right) P_{\infty} \right)$ then $g=f_{d,0}$ for some $d\in \mathbb{F}_r$ as seen above. Assume now 
\[g\in \mathscr{L}\left( \frac{q^3}{r} P_{\infty} \right) \bigg\backslash \mathscr{L}\left( \left( \frac{q^3}{r}-1 \right) P_{\infty} \right).\]
Then, $v_{P_\infty}(g)=q^3/r$ and $g$ has a unique term $\beta x^{\frac{q^2}{r}}$ with order $q^3/r$ at $P_\infty$, $\beta \in \mathbb{F}_{q^2}^*$. Define $\alpha \in \mathbb{F}_{q^2}$ by $\alpha^{r^{2m-1}}=\beta$. Then, $g-f_{0,\alpha} \in \mathscr{L}\left(\frac{q^3}{r} P_{\infty}\right)$ and arguing as in the first part of the proof, shows that $g-f_{0,\alpha}$  is again a constant $d\in \mathbb{F}_r$. This means $g=f_{d,\alpha}$, and the result follows. 
\end{proof}

Similar computation gives that for $\alpha \in \mathbb{F}_{q^2}$, 
\[\tr(\alpha y)\in \mathscr{L}\left( \frac{(q+1)q^2}{r} P_{\infty} \right).\]
Hence, $\dim C_{q,r}((q+1)q^2/r) \geq 4m+1$. By \cite[Table 2]{pinero2014subfield}, we have equality for $q=4$ and $r=2$. Using our GAP package HERmitian \cite{HERmitian01}, we computed the true dimension of $C_{8,2}(s)$ for all values $s$ from $256=q^3/r$ to $511=q^3 -1$, see Table \ref{tab:1}.

\begin{table}
\caption{Parameters of $C_{8,2}(s)$ for $s\in\{256,\ldots,511\}$ \label{tab:1}}
\begin{tabular}{|c|c|c|c|c|c|c|}
\hline 
$s$ & $\dim C_{8,2}(s)$& $\dim \mathcal{H}(64,s)$ & & $s$ & $\dim C_{8,2}(s)$& $\dim \mathcal{H}(64,s)$\\ 
\hline \hline 
256 & 7 & 229 &  & 456 & 206 & 429 \\ \hline 
288 & 13 & 261 &  & 457 & 212 & 430 \\ \hline 
292 & 19 & 265 &  & 458 & 218 & 431 \\ \hline 
320 & 25 & 293 &  & 460 & 224 & 433 \\ \hline 
324 & 28 & 297 &  & 462 & 226 & 435 \\ \hline 
328 & 34 & 301 &  & 464 & 232 & 437 \\ \hline 
336 & 36 & 309 &  & 466 & 238 & 439 \\ \hline 
352 & 42 & 325 &  & 468 & 244 & 441 \\ \hline 
356 & 48 & 329 &  & 470 & 250 & 443 \\ \hline 
360 & 54 & 333 &  & 472 & 256 & 445 \\ \hline 
364 & 60 & 337 &  & 473 & 262 & 446 \\ \hline 
368 & 66 & 341 &  & 474 & 268 & 447 \\ \hline 
376 & 72 & 349 &  & 475 & 274 & 448 \\ \hline 
378 & 74 & 351 &  & 480 & 280 & 453 \\ \hline 
384 & 80 & 357 &  & 482 & 286 & 455 \\ \hline 
392 & 86 & 365 &  & 484 & 292 & 457 \\ \hline 
400 & 92 & 373 &  & 486 & 295 & 459 \\ \hline 
402 & 98 & 375 &  & 488 & 301 & 461 \\ \hline 
408 & 104 & 381 &  & 489 & 307 & 462 \\ \hline 
410 & 110 & 383 &  & 490 & 313 & 463 \\ \hline 
416 & 116 & 389 &  & 491 & 319 & 464 \\ \hline 
418 & 122 & 391 &  & 492 & 325 & 465 \\ \hline 
420 & 128 & 393 &  & 493 & 331 & 466 \\ \hline 
424 & 134 & 397 &  & 496 & 337 & 469 \\ \hline 
428 & 140 & 401 &  & 498 & 343 & 471 \\ \hline 
432 & 146 & 405 &  & 500 & 349 & 473 \\ \hline 
434 & 152 & 407 &  & 502 & 355 & 475 \\ \hline 
436 & 158 & 409 &  & 504 & 361 & 477 \\ \hline 
438 & 164 & 411 &  & 505 & 367 & 478 \\ \hline 
440 & 170 & 413 &  & 506 & 373 & 479 \\ \hline 
442 & 176 & 415 &  & 507 & 379 & 480 \\ \hline 
444 & 182 & 417 &  & 508 & 385 & 481 \\ \hline 
448 & 188 & 421 &  & 509 & 391 & 482 \\ \hline 
450 & 194 & 423 &  & 510 & 397 & 483 \\ \hline 
452 & 200 & 425 &  & 511 & 403 & 484 \\ \hline 
\end{tabular} 
\end{table}

\medskip
\bibliographystyle{abbrv}
\bibliography{Codes}
\end{document}